\documentclass[11pt,reqno]{amsart}

\newtheorem{proposition}{Proposition}[section]

\newtheorem{corollary}[proposition]{Corollary}
\newtheorem{theorem}[proposition]{Theorem}

\theoremstyle{definition}
\newtheorem{definition}[proposition]{Definition}
\newtheorem{example}[proposition]{Example}

\newtheorem{remark}[proposition]{Remark}
\newtheorem{remarks}[proposition]{Remarks}

\newcommand{\thlabel}[1]{\label{th:#1}}
\newcommand{\thref}[1]{Theorem~\ref{th:#1}}
\newcommand{\selabel}[1]{\label{se:#1}}

\newcommand{\colabel}[1]{\label{co:#1}}
\newcommand{\coref}[1]{Corollary~\ref{co:#1}}
\newcommand{\relabel}[1]{\label{re:#1}}
\newcommand{\reref}[1]{Remark~\ref{re:#1}}
\newcommand{\exlabel}[1]{\label{ex:#1}}
\newcommand{\exref}[1]{Example~\ref{ex:#1}}
\newcommand{\delabel}[1]{\label{de:#1}}

\newcommand{\eqlabel}[1]{\label{eq:#1}}
\newcommand{\equref}[1]{(\ref{eq:#1})}

\newcommand{\Cc}{\mathcal{C}}

\def\*C{{}^*\hspace*{-1pt}{\Cc}}
\def\text#1{{\rm {\rm #1}}}

\input xy
\xyoption {all} \CompileMatrices

\usepackage{amssymb}
\usepackage{color,amssymb,graphicx,amscd,amsmath}
\usepackage[colorlinks,urlcolor=blue,linkcolor=blue,citecolor=blue]{hyperref}

\begin{document}
\title[Metabelian associative algebras]
{Metabelian associative algebras}

\author{G. Militaru}
\address{Faculty of Mathematics and Computer Science, University of Bucharest, Str.
Academiei 14, RO-010014 Bucharest 1, Romania}
\email{gigel.militaru@fmi.unibuc.ro and gigel.militaru@gmail.com}

\thanks{This work was supported by a grant of the Romanian National
Authority for Scientific Research, CNCS-UEFISCDI, grant no.
88/05.10.2011.}


%\maketitle

\begin{abstract}
Metabelian algebras are introduced and it is shown that an algebra
$A$ is metabelian if and only if $A$ is a nilpotent algebra having
the index of nilpotency at most $3$, i.e. $x y z t = 0$, for all
$x$, $y$, $z$, $t \in A$. We prove that the It\^{o}'s theorem for
groups remains valid for associative algebras. A structure theorem
for metabelian algebras is given in terms of pure linear algebra
tools and their classification from the view point of the
extension problem is proven. Two border-line cases are worked out
in detail: all metabelian algebras having the derived algebra of
dimension $1$ (resp. codimension $1$) are explicitly described and
classified. The algebras of the first family are parameterized by
bilinear forms and classified by their homothetic relation. The
algebras of the second family are parameterized by the set of all
matrices $(X, Y, u) \in {\rm M}_{n}(k)^2 \times k^n$ satisfying
$X^2 = Y^2 = 0$, $XY = YX$ and $Xu = Yu$.
\end{abstract}

\maketitle

\section*{Introduction}
The concept of a metabelian group goes back to W.B. Fite
\cite{fite} and since then they became a very important topic of
study within group theory \cite{lennox}. At the level of Lie, or
more general Leibniz algebras, the corresponding concept of
metabelian Lie/Leibniz algebra, as $2$-step solvable algebra, is
also well known \cite{am-2014, dani, Dre}. In this paper we
introduce the associative algebra counterpart of a metabelian
group by defining a \emph{metabelian algebra} over a field $k$ as
an extension of an abelian algebra by an abelian algebra - the
word 'abelian' is borrowed from Lie algebras, i.e. an algebra $A$
having the trivial multiplication: $x y = 0$, for all $x$, $y\in
A$. At first sight, such a restrictive definition seems to have
limited chances of leading to an interesting theory of these
associative algebras and, moreover, few of them can be expected to
exist. We shall prove the contrary and, to begin with, we shall
introduce the following argument: the classification of all
metabelian associative algebras having the derived algebra of
dimension $1$ (hence only a small and apparently unattractive
class of algebras) is equivalent to the classification of bilinear
forms on a vector space up to the homothetic relation on bilinear
forms - a relation that generalizes the classical isometric
relation on bilinear forms \cite{will}. On the other hand, the
class of all $(n+1)$-dimensional metabelian associative algebras
having the derived algebra of dimension $n$ is parameterized by an
interesting set of matrices: namely, the matrices $(X, Y, u) \in
{\rm M}_{n}(k)^2 \times k^n$ satisfying $X^2 = Y^2 = 0$, $XY = YX$
and $Xu = Yu$. We prove that an algebra $A$ is metabelian if and
only if $A$ is a nilpotent algebra having the index of nilpotency
at most $3$. The classification of all nilpotent algebras of a
given dimension is a classical problem in the theory of
associative algebras: see \cite[Chapter VI]{KP} where the
classification of nilpotent algebras of dimension $3$ and $4$ is
obtained. The classification of nilpotent algebras of a dimension
higher than $4$ is a difficult problem since the complexity of the
computations increases very rapidly with the dimension
\cite{graaf}. Thus, one has higher chances to succeed in
classifying metabelian algebras of a given dimension instead of
nilpotent ones. The first bridge to solving the problem is given
by \thref{carmatab}, where a structure theorem is proposed. We
show that any metabelian algebra $A$ is isomorphic to an algebra
of the form $P \star V$ associated to a system $(P, \, V, \,
\triangleleft, \, \triangleright, \, \theta)$ consisting of two
vector spaces $P$, $V$ and three bilinear maps $\triangleleft: V
\times P \to V$, $\triangleright: P \times V \to V$, $\theta: P
\times P \to V$ satisfying the following compatibility conditions
for any $p$, $q$, $r\in P$ and $x\in V$:
$$
p \triangleright (x \triangleleft q) = (p \triangleright x)
\triangleleft q, \quad (x \triangleleft p) \triangleleft q = p
\triangleright (q \triangleright x) = 0, \quad p \triangleright \theta (q, \, r) = \theta (p, \, q) \triangleleft
r
$$
We denoted $P \star V := P \times V$ with the multiplication given
for any $p$, $q\in P$, $x$, $y \in V$ by
$$
(p, x) \star (q, y) := (0, \, \theta (p, q) + p \triangleright y +
x \triangleleft q)
$$
Based on this, the classification of all metabelian algebras that
are extensions of a given abelian algebra $P_0$ by an abelian
algebra $V_0$ is obtained in \thref{clascoho} where the explicit
description of the classifying object ${\rm Ext} (P_0, V_0)$ is
given. ${\rm Ext} (P_0, V_0)$ classifies all metabelian algebras
from the viewpoint of the extension problem \cite{Hoch2}, i.e. up
to an isomorphism of algebras that stabilizes $V_0$ and
co-stabilizes $P_0$. In \thref{codim1totul} all metabelian
algebras having the derived algebra of dimension $1$ are
explicitly described, classified and the automorphism groups of
these algebras are determined. The algebras of this family are
classified by bilinear forms on a vector space $P$ up to the
following equivalence relation: two bilinear forms $\theta$ and
$\theta' \in {\rm Bil} \, (P \times P, \, k)$ are called
homothetic \cite{murr} if there exists a pair $(u, \psi) \in k^*
\times {\rm Aut}_k (P)$ such that $u \, \theta (p, \, q) = \theta'
\bigl( \psi(p), \, \psi (q) \bigl)$, for all $p$, $q\in P$. If $k
= k^2$ this relation is equivalent to the classical classification
of bilinear forms \cite{will}. Examples are given in \coref{n23}.
\thref{codim1totulb} and \exref{excodim1} addresses the dual case:
all $(n+1)$-dimensional metabelian algebras having the derived
algebras of codimension $1$ are explicitly described and
classified by the set of all matrices $(X, Y, u) \in {\rm
M}_{n}(k)^2 \times k^n$ satisfying $X^2 = Y^2 = 0$, $XY = YX$ and
$Xu = Yu$.

\section{Preliminaries}\selabel{prel}
All vector spaces, algebras, linear or bilinear maps are over an
arbitrary field $k$. For a vector space $V$ we denote by ${\rm
Bil} \, (V \times V, \, k)$ all bilinear forms on $V$. Two
bilinear forms $\theta$ and $\theta' \in {\rm Bil} \, (V \times V,
\, k)$ are called \emph{isometric}, and we denote this by $\theta
\approx \theta'$, if there exists an automorphism $\varphi \in
{\rm Aut}_k (V)$ such that $\theta (x, y) = \theta'( \varphi (x),
\, \varphi(y))$, for all $x$, $y\in V$. If $V$ is finite
dimensional having $\{e_1, \cdots, e_n\}$ as a basis we also
denote by $\theta = (\theta (e_i, \, e_j)) \in {\rm M}_n(k)$ the
matrix associated to a bilinear form $\theta \in {\rm Bil} \, (V
\times V, \, k)$. Then $\theta \approx \theta'$ if and only if
there exists an invertible matrix $C \in \mathfrak{gl} (n, k)$,
such that $\theta = C^T \, \theta' \, C$ - where $C^T$ is the
transposed of $C$. For future references to the classification
problem of bilinear forms up to an isometry see \cite{horn, rie}
and the references therein. Two bilinear forms $\theta$ and
$\theta' \in {\rm Bil} \, (V \times V, \, k)$ are called
\emph{homothetic} \cite{murr}, and we denote this by $\theta
\equiv \theta'$, if there exists a pair $(u, \psi) \in k^* \times
{\rm Aut}_k (P)$ such that $u \, \theta (p, \, q) = \theta' \bigl(
\psi(p), \, \psi (q) \bigl)$, for all $p$, $q\in P$. Any isometric
bilinear forms are homothetic and if $k = k^2 := \{x^2 \, | \,
x\in k \}$ then any homothetic bilinear forms are isometric.

By an algebra we always mean an associative algebra, i.e. a pair
$(A, M_A)$ consisting of a vector space $A$ and a bilinear map
$M_A: A \times A \to A$, called the multiplication on $A$ and
denoted by $M_A \, (x, y) = xy$, such that $x (y z) = (x y) z$,
for all $x$, $y$, $z\in A$. Any vector space $V$ is an algebra
with the trivial multiplication $x y = 0$, for all $x$, $y\in V$
-- such an algebra is called \emph{abelian} and will be denoted by
$V_0$. ${\rm Aut}_{\rm Alg} \, (A)$ will denote the group of
algebra automorphisms of $A$. An algebra $A$ is called
\emph{nilpotent} if there exists a positive integer $n$ such that
$x_1 x_2 \cdots x_n = 0$, for all $x_1, \cdots, x_n \in A$. If
there is a non-zero product of $n-1$ elements of $A$, then $n$ is
called the \emph{index of nilpotency} of $A$. For an algebra $A$
we denote by $A'$ the \emph{derived subalgebra} of $A$, i.e. the
subspace of $A$ generated by all $xy$, for any $x$, $y \in A$. Let
$P$ and $V$ be two given algebras. An extension of $P$ by $V$ is
triple $(A, \, i, \, \pi)$ consisting of an algebra $A$ and two
morphisms of algebras $i: V \to A$, $\pi: A \to P$ such that
\begin{eqnarray*}
\xymatrix{ 0 \ar[r] & V \ar[r]^{i} & {A} \ar[r]^{\pi} & P \ar[r] &
0 }
\end{eqnarray*}
is an exact sequence. Two extensions $(A, i, \pi)$ and $(A', i',
\pi')$ of $P$ by $V$ are called \emph{equivalent} and we denote
this by $(A, i, \pi) \approx (A', i', \pi)$ if there exists a
morphism of algebras $\varphi: A \to A'$ that stabilizes $V$ and
co-stabilizes $P$, i.e. the following diagram
\begin{eqnarray*}
\xymatrix {& V \ar[r]^{i} \ar[d]_{Id} & {A}
\ar[r]^{\pi} \ar[d]^{\varphi} & P \ar[d]^{Id}\\
& V \ar[r]^{i'} & {A'}\ar[r]^{\pi'} & P}
\end{eqnarray*}
is commutative. Any such morphism $\varphi$ is an isomorphism and
thus $\approx$ is an equivalence relation on the class of all
extensions of $P$ by $V$. We denote by ${\rm Ext} \, (P, \, V)$
the set of all equivalence classes of all extensions of $P$ by $V$
via $\approx$.

\section{Metabelian algebras}\selabel{metab}
We introduce the concept of metabelian associative algebras having
in mind one of the equivalent definitions of metabelian groups.

\begin{definition} \delabel{metabdef}
An algebra $A$ is called \emph{metabelian} if it is an extension
of an abelian algebra by an abelian algebra, that is there exist
two vector spaces $V$, $P$ and an exact sequence of morphisms of
algebras\footnote{We recall that for a vector space $V$, we denote
$V_0 = V$ with the abelian algebra structure.}
\begin{eqnarray} \eqlabel{aldoileasir}
\xymatrix{ 0 \ar[r] & V_0 \ar[r]^{i} & {A} \ar[r]^{\pi} & P_0
\ar[r] & 0 }
\end{eqnarray}
\end{definition}

Explicit examples will be provided at the end of the paper. In
order to proceed with the characterization and structure theorem
for metabelian algebras some preparatory work is needed.

\begin{definition} \delabel{discreterep}
Let $V$ be a vector space. A \emph{discrete bimodule} over $V$ is
a bimodule over the abelian algebra $V_0$, i.e. a triple $(P,
\triangleleft, \triangleright)$ consisting of a vector space $P$
and two bilinear maps $\triangleleft: V \times P \to V$,
$\triangleright: P \times V \to V$ satisfying the following
compatibility conditions for any $p$, $q\in P$ and $x\in V$:
\begin{equation} \eqlabel{disc1}
(x \triangleleft p) \triangleleft q = p \triangleright (q \triangleright x) = 0,
\qquad p \triangleright (x \triangleleft q) = (p \triangleright x) \triangleleft q
\end{equation}
For a discrete bimodule $(P, \triangleleft, \triangleright)$ over
$V$ a bilinear map $\theta : P \times P \to V$ is called a
\emph{discrete $(\triangleleft, \triangleright)$-cocycle} if for
any $p$, $q$ and $r\in P$ we have:
\begin{equation}\eqlabel{discoc}
p \triangleright \theta (q, \, r) = \theta (p, \, q) \triangleleft
r
\end{equation}
A system $(P, \, V, \, \triangleleft, \, \triangleright, \,
\theta)$ consisting of a vector space $V$, a discrete bimodule
$(P, \triangleleft, \triangleright)$ over $V$ and a discrete
$(\triangleleft, \triangleright)$-cocycle $\theta : P \times P \to
V$ is called a \emph{metabelian datum} of $P$ by $V$. We denote
by ${\rm DZ}^2 \, \bigl((P, \triangleleft, \triangleright), \,
V \bigl)$ the space of all discrete $(\triangleleft,
\triangleright)$-cocycles and by ${\rm Met} \, (P, \, V)$ the set
of all metabelian datums $(P, \, V, \, \triangleleft, \, \triangleright, \, \theta)$
of $P$ by $V$.
\end{definition}

Let $(P, \, V, \, \triangleleft, \, \triangleright, \, \theta) \in
{\rm Met} \, (P, \, V)$ and
$P \star V = P \star_{(\triangleleft, \, \triangleright,
\, \theta)} \, V $ be the vector space $P \times V$ with the
multiplication given for any $p$, $q\in P$, $x$, $y \in V$ by:
\begin{equation} \eqlabel{hoproduct2}
(p, x) \star (q, y) := (0, \, \theta (p, q) + p \triangleright y +
x \triangleleft q)
\end{equation}
As a special case of \cite[Proposition 1.2]{am-2013d} we can easily see that $P \star
V$ with the multiplication given by \equref{hoproduct2} is an
associative algebra. In fact, we can easily show that the multiplication given by \equref{hoproduct2}
is associative if and only if $(P, \, V, \, \triangleleft, \,
\triangleright, \, \theta)$ is a metabelian datum.
$P \star V$ is a metabelian algebra since we have the following exact sequence
of algebra maps
\begin{eqnarray} \eqlabel{extenho1}
\xymatrix{ 0 \ar[r] & V_0 \ar[r]^{i_{V}} & P \star \, V
\ar[r]^{\pi_{P}} & P_0 \ar[r] & 0 }
\end{eqnarray}
where $i_V (x) = (0, x)$ and $\pi_p (p, x) := p$, for all $x\in V$
and $p\in P$. The algebra $P \star V$ is called the
\emph{metabelian product} of $P$ over $V$ associated to $(P, \, V,
\, \triangleleft, \, \triangleright, \, \theta) \in {\rm Met} \,
(P, \, V)$. The next characterization and structure theorem for
metabelian algebras shows that any metabelian algebra $A$ is
isomorphic to some $P \star V$.

\begin{theorem}\thlabel{carmatab}
For an associative algebra $A$ the following statements are
equivalent:

$(1)$ $A$ is metabelian;

$(2)$ $A$ is a nilpotent algebra having the index of nilpotency at
most $3$, i.e. $x y z t = 0$, for all $x$, $y$, $z$, $t \in A$;

$(3)$ The derived algebra $A'$ is an abelian subalgebra of $A$;

$(4)$ There exists an isomorphism of algebras $A \cong P\star V$,
for some vector spaces $P$ and $V$ and for some $(P, \, V, \,
\triangleleft, \, \triangleright, \, \theta) \in {\rm Met} \, (P,
\, V)$.
\end{theorem}

\begin{proof}
$(3)$ is just an equivalent rephrasing of $(2)$, i.e. $(2)
\Leftrightarrow (3)$. The exactness of the sequence
\equref{extenho1} proves that $P \star V$ is metabelian, that is
$(4) \Rightarrow (1)$. Now, if $A'$ is an abelian subalgebra of
$A$, we have an exact sequence of algebra maps
\begin{eqnarray} \eqlabel{extenho2b}
\xymatrix{ 0 \ar[r] & A' = A'_0 \ar[r]^{i} & A \ar[r]^{\pi} &
(A/A')_0 \ar[r] & 0 }
\end{eqnarray}
where $i$ is the inclusion map and $\pi$ the canonical projection.
This proves $(3) \Rightarrow (1)$. We prove now that $(1)
\Rightarrow (2)$. Assume that \equref{aldoileasir} is an exact
sequence of algebra maps. We can assume that $V$ is a subspace of
$A$. The fact that $i$ is an algebra map shows that $x y = 0$, for
all $x$, $y\in V$. On the other hand, the fact that $\pi$ is a
morphism of algebras and the exactness of the sequence
\equref{aldoileasir} translates to $\pi ( a b) = 0$, i.e. $a b \in
{\rm Ker} (\pi) = V$, for all $a$, $b\in A$. Thus, the product
$xyzt = 0$, for all $x$, $y$, $z$ and $t \in A$ since the product
of any two elements of $V$ is zero. If we show $(1) \Rightarrow
(4)$ the proof is finished. Since $k$ is a field we can pick a
$k$-linear section $s : P \to A$ of $\pi$, i.e. $\pi \circ s =
{\rm Id}_{P}$. Using the section $s$ we define three bilinear maps
$\triangleleft_{s} \, : V \times P \to V$, $\triangleright_{s} \,
: P \times V \to V$ and $\theta_s \, : P \times P \to V$ by the
following formulas:
\begin{eqnarray*}
x \triangleleft p := x s(p) \eqlabel{act1}, \qquad p
\triangleright x := s(p) x, \qquad \theta (p, q) := s(p) s(q)
\end{eqnarray*}
for all $p$, $q\in P$ and $x\in V$. Then, by a straightforward
computation (or as a special case of \cite[Proposition
1.4]{am-2013d}) we can show that $(P, \, V, \, \triangleleft_s, \,
\triangleright_s, \, \theta_s) \in {\rm Met} \, (P, \, V)$ and the
map
\begin{equation} \eqlabel{izomor}
\varphi : P \star V \to A, \qquad \varphi (p, x) := s(p) + x
\end{equation}
is an isomorphism of algebras with $\varphi^{-1} (y) =
(\pi(y), \, y - s (\pi(y)) )$, for all $y\in A$.
\end{proof}

One of the fundamental results in the factorization theory for
groups is the famous It\^{o}'s theorem \cite{ito} which has been
the key ingredient in proving many structural theorems for finite
groups \cite{AFG}: if $G$ is a group such that $G = AB$, for two
abelian subgroups $A$ and $B$, then $G$ is metabelian. As a
special case of \thref{carmatab} we obtain the counterpart of
It\^{o}'s theorem for associative algebras:

\begin{corollary} \colabel{itoass} \textbf{(It\^{o}'s theorem for associative
algebras)} Let $A$ be an associative algebra such that $A = P_{0}
+ V_{0}$, for two abelian subalgebras $P_{0}$ and $V_{0}$ of $A$.
Then $A$ is metabelian.
\end{corollary}

\begin{proof}
Using \thref{carmatab} we have to prove that $a_{1} \, a_{2} \,
a_{3} \, a_{4} = 0$, for all $a_{i} \in A$, $i = 1, \cdots, 4$.
Indeed, since $a_{i} \in A = P_{0} + V_{0}$ we can find $v_{i} \in
V_{0}$ and $p_{i} \in P_{0}$ such that $a_{i} = v_{i} + p_{i}$,
for all $i = 1, \cdots, 4$. Let $v_{5}$, $v_{6} \in V_{0}$ and
$p_{5}$, $p_{6} \in P_{0}$ such that $p_{2} \, v_{3} = v_{5} +
p_{5}$ and $v_{2} \, p_{3} = v_{6} + p_{6}$. Using intensively
that $P_{0}$ and $V_{0}$ are both abelian we obtain:
\begin{eqnarray*}
a_{1} \, a_{2} \, a_{3} \, a_{4} &=& (\underline{v_{1}\, v_{2}} +
v_{1}\, p_{2} + p_{1}\,v_{2} + \underline{p_{1} \, p_{2}})
(\underline{v_{3}\, v_{4}} + v_{3}\,
p_{4} + p_{3}\,v_{4} + \underline{p_{3} \, p_{4}})\\
&=& (v_{1}\, p_{2} + p_{1}\,v_{2}) (v_{3}\, p_{4} + p_{3}\,v_{4})\\
&=& v_{1}\,p_{2}\,v_{3}\,p_{4} + v_{1}\,\underline{p_{2}\,p_{3}}\,
v_{4} + p_{1}\,\underline{v_{2}\,v_{3}}\,p_{4} + p_{1}\,v_{2}\, p_{3}\, v_{4}\\
&=& v_{1}\,p_{2}\,v_{3}\,p_{4} + p_{1}\,v_{2}\, p_{3}\, v_{4} \, =
\, v_{1} \, (v_{5} + p_{5}) \, p_{4} + p_{1}\,(v_{6} + p_{6})\,v_{4}\\
&=& \underline{v_{1}\,v_{5}}\,p_{4} +
v_{1}\,\underline{p_{5}\,p_{4}} + p_{1}\,\underline{v_{6}\,v_{4}}
+ \underline{p_{1}\,p_{6}}\,v_{4} \, = \,  0
\end{eqnarray*}
as needed.
\end{proof}

\begin{remark}\relabel{clascoh}
There is more to be said related to the proof of \thref{carmatab}:
if $A$ is a metabelian algebra, extension of $P_0$ by $V_0$, then
the isomorphism $\varphi: P \star V \to A$ given by
\equref{izomor} stabilizes $V$ and co-stabilizes $P$, i.e. the
diagram
\begin{eqnarray*}
\xymatrix {& V_0 \ar[r]^{i_V} \ar[d]_{Id} & {P\star V}
\ar[r]^{\pi_p} \ar[d]^{\varphi} & P_0 \ar[d]^{Id}\\
& V_0 \ar[r]^{i} & {A}\ar[r]^{\pi} & P_0}
\end{eqnarray*}
is commutative. Hence, the extension $(A, i, p)$ of $P_0$ by $V_0$
is equivalent to the extension $(P\star V, i_V, \pi_P)$ of $P_0$
by $V_0$ given by \equref{extenho1}. Thus the classification of
metabelian algebras is reduced to the classification of all
metabelian products of $P\star V$ associated to all $(P, \, V, \,
\triangleleft, \, \triangleright, \, \theta) \in {\rm Met} \, (P,
\, V)$. This will be given below by indicating the explicit
description of the classifying object ${\rm Ext} (P_0, V_0)$.
\end{remark}

Let $(P, \triangleleft, \triangleright)$ be a fixed discrete
bimodule over a vector space $V$. Two discrete $(\triangleleft,
\triangleright)$-cocycles $\theta$, $\theta': P\times P \to V$ are
called \emph{cohomologous} and we denote this by $\theta \equiv
\theta'$ if there exists a linear map $r: P \to V$ such that
\begin{equation}\eqlabel{discocho}
\theta(p, \, q) = \theta' (p, \, q) + p \triangleright r(q) + r(p)
\triangleleft q
\end{equation}
for all $p$, $q\in P$. As a special case of \cite[Lemma 1.5 and
Definition 1.6]{am-2013d} we can prove that $\equiv$ is an
equivalence relation on ${\rm DZ}^2 \, \bigl((P, \triangleleft,
\triangleright), \, V \bigl)$ and denote by
$$
{\rm DH}^2 \, \bigl((P, \triangleleft, \triangleright), \, V
\bigl) \, := \, {\rm DZ}^2 \, \bigl((P, \triangleleft,
\triangleright), \, V \bigl)/\equiv
$$
the quotient set via $\equiv$, called the \emph{discrete
cohomological group.} In fact, \cite[Lemma 1.5]{am-2013d}, applied
for the abelian case, shows that two extensions of $P_0$ by $V_0$
of the form $(P\star_{(\triangleleft, \, \triangleright, \,
\theta)} V, i_V, \pi_P)$ and $(P\star_{(\triangleleft', \,
\triangleright', \, \theta)} V, i_V, \pi_P)$ are equivalent if and
only if $\triangleleft' = \triangleleft$, $\triangleright' =
\triangleright$ and $\theta' \equiv \theta$. All these
considerations lead to the following classification result that
gives the decomposition of ${\rm Ext} (P_0, V_0)$ as the
co-product of all discrete cohomological groups.

\begin{theorem}\thlabel{clascoho}
Let $V$ and $P$ be two vector spaces. Then there exists a
bijection
\begin{equation} \eqlabel{formulamare}
{\rm Ext} (P_0, V_0) \cong \, \sqcup_{(\triangleleft,
\triangleright)} \, {\rm DH}^2 \, \bigl((P, \triangleleft,
\triangleright), \, V \bigl)
\end{equation}
where $\sqcup_{(\triangleleft, \triangleright)}$ is the co-product
in the category of sets over all possible discrete bimodule
structures $(P, \triangleleft, \triangleright)$ over $V$. The
explicit bijection sends an element $\overline{\theta} \in {\rm
DH}^2 \, \bigl((P, \triangleleft, \triangleright), \, V \bigl)$ to
the metabelian product $ P\star_{(\triangleleft, \,
\triangleright, \, \theta)} \, V$, where $\overline{\theta}$
denotes the equivalence class of the discrete $(\triangleleft,
\triangleright)$-cocycle $\theta$ via $\equiv$.
\end{theorem}

\begin{remarks} \relabel{spargerea}
$(1)$ The formula \equref{formulamare} highlights an algorithm
that breaks up the problem of computing ${\rm Ext} (P_0, V_0)$
into three steps: (1) first, all discrete bimodule structures $(P,
\triangleleft, \triangleright)$ over $V$ are described; (2) then,
for a given structure $(P, \triangleleft, \triangleright)$, we
compute the space ${\rm DZ}^2 \, \bigl((P, \triangleleft,
\triangleright), \, V \bigl)$ of all discrete $(\triangleleft,
\triangleright)$-cocycles; (3) in the last step all quotient
spaces ${\rm DH}^2 \, \bigl((P, \triangleleft, \triangleright), \,
V \bigl) = {\rm DZ}^2 \, \bigl((P, \triangleleft, \triangleright),
\, V \bigl)/\approx$ are described and their co-product is
computed. Specific examples are given below.

$(2)$ If we are interested in the classification of all metabelian
algebras $A$ of a given dimension then the key object is the
derived algebra $A'$ since \thref{carmatab} shows that $A \cong P
\star A'$, where $P := A/A'$, the quotient vector space viewed as
an abelian algebra. Thus, there are two numbers involved in the
problem of classifying finite dimensional metabelian algebras:
$n$, the dimension of the metabelian algebras that we are looking
for, and $m := {\rm dim}_k (A') \leq n$. Two border-line cases are
immediately settled: if $m = n$, then $A' = A$, and thus we
obtain, using $(3)$ of \thref{carmatab}, that $A$ is the abelian
algebra $k^n_0$. On the other hand, if $m = 0$, that is $A' = 0$,
then $A$ is also the abelian algebra $k^n_0$, by the definition of
$A'$. The next two steps cover the cases when $m = 1$ (resp. $m =
n-1$).
\end{remarks}

First we shall describe and classify metabelian algebras having
the derived algebra of dimension $1$; the group of algebra
automorphisms of such algebras is also determined.

\begin{theorem}\thlabel{codim1totul}
Let $P$ be a vector space. Then:

$(1)$ ${\rm Ext} \, (P_0, \, k_0) \cong {\rm Bil} \, (P \times P,
\, k)$ and the equivalence classes of all metabelian algebras that
are extensions of $P_0$ by $k_0$ are represented by the extensions
of the form
\begin{eqnarray} \eqlabel{codim1gener}
\xymatrix{ 0 \ar[r] & k_0 \ar[r] & P_{\theta} := P \star_{\theta}
\, k \ar[r] & P_0 \ar[r] & 0 }
\end{eqnarray}
for any $\theta \in {\rm Bil} \, (P \times P, \, k)$; the
multiplication on the algebra $P_{\theta}$ is given by: $ (p, \,
x) \star (q, \, y) = (0, \, \theta (p, \, q) )$, for all $p$,
$q\in P$, $x$, $y \in k$.

Any metabelian algebra having the derived algebra of dimension $1$
is isomorphic to a $P_{\theta}$, for some vector space $P$ and
$\theta \in {\rm Bil} \, (P \times P, \, k)$.

$(2)$ Two algebras $P_{\theta}$ and $P_{\theta'}$ are isomorphic
if and only if the bilinear forms $\theta$ and $\theta$ are
homothetic, i.e. there exists a pair $(u, \psi) \in k^* \times
{\rm Aut}_k (P)$ such that for all $p$, $q\in P$
\begin{equation}\eqlabel{clasificgenfor}
u \, \theta (p, \, q) = \theta' \bigl( \psi(p), \, \psi (q)
\bigl)
\end{equation}
In particular, if $k = k^2$ then, $P_{\theta} \cong P_{\theta'}$
if and only if $\theta$ and $\theta'$ are isometric.

$(3)$ The group of algebra automorphisms ${\rm Aut}_{\rm Alg} \,
(P_{\theta})$ is isomorphic to
$$
{\mathcal G} (P, \theta) := \{ (u, \lambda, \psi) \in k^* \times
P^* \times {\rm Aut}_k (P) \, | \,\, u \, \theta (p, \, q) =
\theta \bigl( \psi(p), \, \psi (q) \bigl), \, \forall \, p, q \in
P \}
$$
where ${\mathcal G} (P, \theta)$ is a group with respect to the
following multiplication:
\begin{equation} \eqlabel{grupul}
(u, \lambda, \psi) \cdot (u', \lambda', \psi') := (uu', \, \lambda
\circ \psi' + u \lambda', \, \psi \circ \psi')
\end{equation}
for all $(u, \lambda, \psi)$ and $(u', \lambda', \psi') \in
{\mathcal G} (P, \theta)$.
\end{theorem}

\begin{proof}
$(1)$ The proof is based on \thref{carmatab} and \thref{clascoho}
following the three steps described in \reref{spargerea} applied
in the case that $V := k$. Using the first compatibility of
\equref{disc1} we can easily prove that $P$ has only one structure
of a discrete bimodule over $k$, namely the trivial one: $ x
\triangleleft p = p \triangleright x :=0$, for all $x\in k$ and $p
\in P$. Hence, the set of all discrete $(\triangleleft,
\triangleright)$-cocycles $\theta : P \times P \to k$ is precisely
the space of all bilinear forms of $P$ since \equref{discoc} holds
trivially. Moreover, the equivalence relation \equref{discocho} is
just the equality between the two bilinear maps. Thus, using the
decomposition formula given by \equref{formulamare}, we obtain
that ${\rm Ext} \, (P_0, \, k_0) \cong {\rm Bil} \, (P \times P,
\, k)$. The last statement follows from \thref{carmatab}.

$(2)$ If $\theta = 0$ (the trivial bilinear form), then $P_{\theta
= 0}$ is the abelian algebra; thus, if $P_{\theta = 0} \cong
P_{\theta'}$, then $\theta' = 0$ and there is nothing to prove. We
will assume that $\theta \neq 0$. We prove a more general
statement which shall also provide the proof of $(3)$. More
precisely, for two non-trivial bilinear forms $\theta$ and
$\theta'$ on $P$ we shall prove that there exists a bijection
between the set of all isomorphisms of algebras $\varphi :
P_{\theta} \to P_{\theta'}$ and the set of all triples $(u,
\lambda, \psi) \in k^* \times {\rm Hom}_k (P, \, k) \times {\rm
Aut}_k (P)$ satisfying the compatibility condition
\equref{clasificgenfor}. Moreover, the bijection is given such
that the isomorphism $\varphi = \varphi_{(u, \lambda, \psi)} :
P_{\theta} \to P_{\theta'}$ corresponding to $(u, \lambda, \psi)$
is given by
\begin{equation}\eqlabel{formulamor}
\varphi (p, \, x) = (\psi (p), \, \lambda(p) + u\,x)
\end{equation}
for all $p\in P$ and $x\in k$. Indeed, any $k$-linear map $\varphi
: P \times k \to P \times k$ is uniquely determined by a quadruple
$(u, \, \lambda, \, \psi, \, p_0) \in k\times {\rm Hom}_k (P, \,
k) \times {\rm End}_k (P) \times P$ such that
$$
\varphi (p, \, x) = \varphi_{(u, \, \lambda, \, \psi, \, p_0)} \,
(p, \, x) = (\psi (p) + x \, p_0,  \,\, \lambda(p) + x \,u)
$$
for all $p \in P$ and $x\in k$. Now, we can easily see that
$\varphi_{(u, \, \lambda, \, \psi, \, p_0)} : P_{\theta} \to
P_{\theta'}$ is an algebra map if and only if the following two
compatibilities hold
$$
\theta (p, \, q) \, p_0 = 0, \quad  u \, \theta (p, \, q) =
\theta' \bigl( \psi(p) + x\, p_0, \, \psi (q) + y \, p_0\bigl)
$$
for all $p$, $q\in P$. Since $\theta \neq 0$ we obtain that
$\varphi_{(u, \, \lambda, \, \psi, \, p_0)}$ is an algebra map if
and only if $p_0 = 0$ and \equref{clasificgenfor} holds. In what
follows we denote by $\varphi_{(u, \, \lambda, \, \psi)}$ the
algebra map corresponding to a quadruple $(u, \, \lambda, \, \psi,
\, p_0)$ with $p_0 = 0$. It remains to be proven that such a
morphism $\varphi = \varphi_{(u, \, \lambda, \, \psi)}$ is
bijective if and only if $\psi$ is bijective and $u \neq 0$.
Assume first that $\varphi$ is bijective: then its inverse
$\varphi^{-1}$ is an algebra map and thus has the form
$\varphi^{-1} (q, \, y) = (\psi' (q), \, \lambda' (q) + y u')$,
for some triple $(u', \lambda', \psi')$. If we write $\varphi^{-1}
\circ \varphi (0, \, 1) = (0, \, 1)$ we obtain that $u u' = 1$
i.e. $u$ is invertible. In the same way $\varphi^{-1} \circ
\varphi (p, \, 0) = (p, \, 0) = \varphi \circ \varphi^{-1} (p, \,
0)$ gives that $\psi$ is bijective and $\psi' = \psi^{-1}$.
Conversely, if $(u, \psi) \in k^* \times {\rm Aut}_k (P)$ then we
can see that $\varphi_{(u, \, \lambda, \, \psi)}$ is bijective
having the inverse given by $\varphi_{(u, \, \lambda, \,
\psi)}^{-1} := \varphi_{(u^{-1}, \, -\lambda \circ \psi^{-1}, \,
\psi^{-1})}$ as needed. For the last statement we remark that if
$\theta \approx \theta'$ then \equref{clasificgenfor} holds for $u
= 1$. Conversely, if $k = k^2$, then we can write $u = v^2$, for
some $v \in k^*$. Multiplying the equation \equref{clasificgenfor}
by $v^{-2}$ and substituting $\psi$ with $v^{-1} \psi$ we obtain
that $\theta \approx \theta'$.

$(3)$ Follows the proof of $(2)$ once we observe that for two
triples $(u, \lambda, \psi)$ and  $(u', \lambda', \psi') \in k^*
\times P^* \times {\rm Aut}_k (P)$ we have that $\varphi_{(u, \,
\lambda, \, \psi)} \circ \varphi_{(u', \, \lambda', \, \psi')} =
\varphi_{(uu', \, \lambda \circ \psi' + u \lambda', \, \psi \circ
\psi')} $.
\end{proof}

\thref{codim1totul} reduces the classification of all
$(n+1)$-dimensional metabelian algebras having the derived algebra
of dimension $1$ to the classification of bilinear forms on $k^n$
up to the equivalence relation given by \equref{clasificgenfor}.
If $k = k^2$ this is just the classical classification of bilinear
forms solved in \cite{horn} for algebraically closed or real
closed fields.

\begin{example} \exlabel{exdim1}
Let $\{e_1, \cdots, e_n\}$ be the canonical basis of $k^n$.
Applying \thref{codim1totul} for $P = k^n$ we obtain that any
$(n+1)$-dimensional metabelian algebra having the derived algebra
of dimension $1$ is isomorphic to an algebra denoted by
$k^{n+1}_{\theta} := k^n \star_{\theta} k$, for some nontrivial
bilinear form $\theta \in {\rm Bil} \, (k^n \times k^n, \, k)$.
Explicitly, $k^{n+1}_{\theta}$ is the algebra having $\{E, \, F_1,
\cdots, F_n\}$ as a basis and the multiplication defined for any
$i$, $j = 1, \cdots, n$ by:
$$
F_i \star F_j := \theta (e_i, e_j) \, E
$$
undefined multiplications on the elements of the basis are $0$.
Two such algebras $k^{n+1}_{\theta}$ and $k^{n+1}_{\theta'}$ are
isomorphic if and only if there exists a pair $(u, \, C) \in k^*
\times \mathfrak{gl} (n, k)$ such that $u\, \theta = C^T \,
\theta' \, C$, where we write $\theta = (\theta (e_i, \, e_j))$
and $\theta' = (\theta' (e_i, \, e_j)) \in {\rm M}_n (k)$.

Assume that $k$ is algebraically closed of characteristic $\neq
2$. Then, $k^{n+1}_{\theta}$ and $k^{n+1}_{\theta'}$ are
isomorphic if and only if $\theta \approx \theta'$. If $n = 2$ the
equivalence classes of all bilinear forms on $k^2$ are given by
the following two families of matrices \cite{horn}:
$$
\theta_{a, \, b} = \begin{pmatrix} 1 & a \\
b & 0 \end{pmatrix}, \qquad
\theta = \begin{pmatrix} 0 & 1 \\
-1 & 0 \end{pmatrix}
$$
for all $a$, $b\in k$. On the other hand, for $n = 3$ the
equivalence classes of all bilinear forms on $k^3$ are given by
the following six families of matrices \cite{horn} for any $a$,
$b\in k$:
$$
\theta^1_{a, \, b} =
\begin{pmatrix}
1 & 0 & 0 \\
0 & 1 & a \\
0 & b & 0
\end{pmatrix}, \qquad
\theta^2_{a, \, b} =
\begin{pmatrix}
0 & 0 & 0 \\
0 & 1 & a \\
0 & b & 0
\end{pmatrix}, \qquad
\theta^3 =
\begin{pmatrix}
1 & 0 & 0 \\
0 & 0 & 1 \\
0 & -1 & 0
\end{pmatrix}
$$
$$
\theta^4 =
\begin{pmatrix}
0 & 0 & 0 \\
0 & 0 & 1 \\
0 & -1 & 0
\end{pmatrix}, \qquad
\theta^5 =
\begin{pmatrix}
0 & 1 & 0 \\
0 & 0 & 1 \\
0 & -1 & 0
\end{pmatrix}, \qquad
\theta^6 =
\begin{pmatrix}
1 & 1 & 0 \\
0 & 1 & 1 \\
0 & 1 & 0
\end{pmatrix}
$$
\end{example}

To conclude, we obtain the following classification results:

\begin{corollary}\colabel{n23}
Let $k$ be an algebraically closed field of characteristic $\neq
2$. Then:

$(1)$ The isomorphism classes of $3$-dimensional metabelian
algebras having the derived algebra of dimension $1$ are the
following two families of algebras defined for any $a$, $b\in k$:
\begin{eqnarray*}
&k^3_{a, \, b}:& \qquad F_1 \star F_1 = E, \quad F_1 \star F_2 = a \, E, \quad F_2\star F_1 = b\, E\\
&k^3_{-1}:& \qquad F_1 \star F_2 = - F_2 \star F_1 = E
\end{eqnarray*}
$(2)$ The isomorphism classes of $4$-dimensional metabelian
algebras having the derived algebra of dimension $1$ are the
following six families of algebras defined for any $a$, $b\in k$:
\begin{eqnarray*}
&k^{4, 1}_{a, \, b}:& \qquad F_1 \star F_1 = F_2 \star F_2 = E, \quad F_2 \star F_3 = a \, E, \quad F_3\star F_2 = b\, E\\
&k^{4, 2}_{a, \, b}:& \qquad F_2 \star F_2 = E, \quad F_2 \star F_3 = a \, E, \quad F_3\star F_2 = b\, E \\
&k^{4, 3}:& \qquad F_1 \star F_1 = F_2 \star F_3 = - F_3 \star F_2 = E \\
&k^{4, 4}:& \qquad F_2 \star F_3 = - F_3 \star F_2 = E \\
&k^{4, 5}:& \qquad F_1 \star F_2 = F_2 \star F_3 = - F_3 \star F_2 = E \\
&k^{4, 6}:& \qquad F_1 \star F_1 = F_1 \star F_2 = F_2 \star F_2 =
F_2\star F_3 = F_3\star F_2 = E
\end{eqnarray*}
\end{corollary}

Now we shall describe the metabelian algebras having the derived
algebra of codimension $1$. For a vector space $V$ we denote by
${\mathcal T} (V) \subseteq {\rm End}_k (V)^2$ the set of all
pairs of endomorphisms $(\lambda, \Lambda) \in {\rm End}_k (V)
\times {\rm End}_k (V)$ satisfying the following compatibility
condition
\begin{equation} \eqlabel{exp=k1}
\lambda^2 = \Lambda^2 = 0, \qquad \Lambda \circ \lambda = \lambda
\circ \Lambda
\end{equation}

\begin{theorem}\thlabel{codim1totulb}
Let $V$ be a vector space. Then:

$(1)$ There exists a bijection
$$
{\rm Met} \, (k, \, V) \cong \{ \, (\lambda, \, \Lambda, \, \zeta)
\in {\mathcal T} (V) \times V \,\, | \,\, \zeta \in {\rm Ker} (\Lambda
- \lambda) \, \}
$$
given such that the metabelian datum $(k, \, V, \, \triangleleft,
\, \triangleright, \theta)$ of $k$ by $V$ associated to a triple
$(\lambda, \, \Lambda, \, \zeta)$ is given for any $x \in V$ and $p$, $q \in k$ by:
\begin{equation}\eqlabel{50}
x \triangleleft p := p \, \lambda (x) \quad p \triangleright x :=
p \, \Lambda(x), \quad \theta (p, q) := pq \, \zeta
\end{equation}
$(2)$ There exists a bijection
\begin{equation} \eqlabel{codim1gen}
{\rm Ext} \, (k_0, \, V_0) \, \cong \, \sqcup_{(\lambda, \Lambda)
\in {\mathcal T} (V)} \, \bigl( {\rm Ker} (\Lambda -
\lambda)/{\rm Im} (\Lambda + \lambda) \bigl)
\end{equation}
that sends an element $\overline{\zeta} \in  {\rm Ker} (\Lambda -
\lambda)/ {\rm Im} (\Lambda + \lambda) \bigl)$ to the metabelian
product $k \star_{(\lambda, \Lambda, \zeta)} \, V = k\times V$,
with the multiplication given for any $p$, $q\in k$ and $x$, $y
\in V$ by:
\begin{equation} \eqlabel{51}
(p, \, x) \star (q, \, y) = ( 0, \, pq \, \zeta + p \, \Lambda (y) + q \, \lambda(x) )
\end{equation}
$(3)$ Any metabelian algebra having the derived algebra of codimension $1$ is isomorphic to
an algebra $k \star_{(\lambda, \Lambda, \zeta)} \, V$, for some vector space
$V$, $(\lambda, \, \Lambda) \in {\mathcal T} (V)$ and $\zeta
\in {\rm Ker} (\Lambda - \lambda)$.
\end{theorem}

\begin{proof}
$(1)$ We apply the definitions of the metabelian datum for $P: =
k$. First of all we show that there exists a bijection between the
set of all discrete bimodule structures over $V$ on $k$ and
${\mathcal T} (V)$. Indeed, since $P = k$ then any bilinear map
$\triangleleft : V \times k \to V$ (resp. $\triangleright : k
\times V \to V$) is uniquely implemented by a linear map $\lambda
: V \to V$ (resp. $\Lambda : V \to V$) via the formulas
\begin{equation}\eqlabel{bla}
x \triangleleft p := p \lambda (x) \qquad p \triangleright x := p
\Lambda (x)
\end{equation}
for all $x\in V$ and $p\in k$. We can easily see that the pair $(
\triangleleft = \triangleleft_{\lambda}, \, \triangleright
=\triangleright_{\Lambda} )$ satisfies the compatibility condition
\equref{disc1} if and only if \equref{exp=k1} holds. Let now
$(\lambda, \Lambda) \in {\rm End}_k (V)$ be a pair satisfying
\equref{exp=k1} and consider $(k, \triangleleft_{\lambda},
\triangleright_{\Lambda})$ with the discrete bimodule structures
over $V$ implemented by $(\lambda, \Lambda)$ via \equref{bla}.
Then we can see that ${\rm DZ}^2 \, \bigl((k,
\triangleleft_{\lambda}, \triangleright_{\Lambda}), \, V \bigl) \,
\cong \, {\rm Ker} (\Lambda - \lambda)$, and the bijection sends
any $\zeta \in {\rm Ker} (\Lambda - \lambda)$ to the associated
discrete $(\triangleleft_{\lambda}, \, \triangleright_{\Lambda}
)$-cocycle $\theta_{\zeta}$ given by $\theta_{\zeta} (p, q) := pq
\, \zeta$, for all $p$, $q \in P$. Thus, we have proved that there
exists a one to one correspondence between ${\rm Met} \, (k, \,
V)$ and the set of all triples $(\lambda, \, \Lambda, \, \zeta)
\in {\mathcal T} (V) \times V $ such that $\zeta \in {\rm Ker}
(\Lambda - \lambda)$, as needed.

$(2)$ We fix a pair $(\lambda, \Lambda) \in {\mathcal T} (V)$;
using \equref{exp=k1}, we observe first that ${\rm Im} (\Lambda +
\lambda) \leq {\rm Ker} (\Lambda - \lambda)$. Hence, the quotient
vector space from the right hand side of \equref{codim1gen} is
defined. Let $\zeta \in {\rm Ker} (\Lambda - \lambda)$ and $\theta
= \theta_{\zeta}$ be the associated discrete
$(\triangleleft_{\lambda}, \, \triangleright_{\Lambda} )$-cocycle.
Then, we can see that $\theta_{\zeta} \equiv \theta_{\zeta'}$ in
the sense of \equref{discocho} if and only if $\zeta - \zeta' \in
{\rm Im} (\Lambda + \lambda)$. This shows that
$$
{\rm DH}^2 \, \bigl((k, \triangleleft_{\lambda},
\triangleright_{\Lambda}), \, V \bigl) \, \cong \, {\rm Ker}
(\Lambda - \lambda)/ {\rm Im} (\Lambda + \lambda)
$$
and the conclusion follows from \thref{clascoho} and
\thref{carmatab}.
\end{proof}

\begin{example} \exlabel{excodim1}
By taking $V: = k^n$ in \thref{codim1totulb} we obtain the
explicit description of all $(n+1)$-dimensional metabelian
algebras having the derived algebra of dimension $n$. Indeed, we
denote by
$$
{\mathcal T} (n) := \{ \, (X, Y) \in {\rm M}_{n}(k) \, \, | \,\,
X^2 = Y^2 = 0, \, \, \, XY = YX \, \}
$$
For a pair of matrices $(X, Y) \in {\mathcal T} (n)$ we denote by
$ {\mathcal E} (X, Y) := \{ \, u \in k^n \, \, | \,\, X u = Y u
\}$ the equalizer of $X$ and $Y$. Then we have:
$$
{\rm Ext} \, (k_0, \, k^n_0) \, \cong \, \sqcup_{(X, Y) \in
{\mathcal T} (n)} \,\, \bigl( {\mathcal E} (X, Y)/\equiv \bigl)
$$
where $\equiv$ is the following equivalence relation on ${\mathcal
E} (X, Y)$: $u \equiv u'$ if and only if there exists $r\in k^n$
such that $u - u' = (X+Y) r$. For $(X = (x_{ij}), Y = (x_{ij}),
\overline{u} =\overline {(u_1, \cdots, u_n)}) \in {\mathcal T} (n)
\times {\mathcal E} (X, Y)/\equiv$, we denote by $k^{n+1}_{X, Y,
\overline{u}}$ the associated metabelian product $k \star k^n$.
Then $k^{n+1}_{X, Y, \overline{u}}$ is the algebra having $\{F, \,
E_1, \cdots, E_n\}$ as a basis and the multiplication defined for
any $i$, $j = 1, \cdots, n$ by:
$$
F \star F := \sum_{j=1}^n \, u_j \, E_j, \quad F\star E_i :=
\sum_{j=1}^n \, y_{ji} \, E_i, \quad E_i\star F := \sum_{j=1}^n \,
x_{ji} \, E_i
$$
Any $(n+1)$-dimensional metabelian algebra having the derived
algebra of codimension $1$ is isomorphic to such an algebra
$k^{n+1}_{X, Y, \overline{u} }$, for some $(X, Y, \overline{u})
\in {\mathcal T} (n) \times {\mathcal E} (X, Y)/\equiv$.
Classifying these algebras for an arbitrary $n$ is a very
difficult task.
\end{example}

\subsection*{Final comments}
\thref{clascoho} classifies metabelian algebras from the viewpoint
of the extension problem \cite{Hoch2}: for two given vector spaces
$P$ and $V$, ${\rm Ext} (P_0, V_0)$ classifies all metabelian
algebras that are extensions of $P_0$ by $V_0$ up to an
isomorphism of algebras that stabilizes $V_0$ and co-stabilizes
$P_0$. Even if the explicit computation of the classifying object
${\rm Ext} (P_0, V_0)$ offers important information, it is not
enough to classify up to an isomorphism all metabelian algebras of
a given dimension. Based on \thref{carmatab}, the next step is to
ask when two arbitrary metabelian products $P \star V$ and $P'
\star V'$ are isomorphic as algebras. Mutatis-mutandis, this is
the associative algebra version of the isomorphism problem from
metabelian groups, which is a very difficult question that seems
to be connected to Hilbert's Tenth problem, i.e. the problem is
algorithmically undecidable (cf. \cite{baums}).
\thref{codim1totul} gives the full answer in the particular case
$V = V' := k$: the isomorphism of two algebras $P \star k$ and $P'
\star k$ is equivalent to classifying the bilinear forms as stated
there. Unfortunately, a similar result offering a necessary and
sufficient criterion for two metabelian products $k \star V$ and
$k' \star V'$ from \thref{codim1totulb} to be isomorphic could not
be obtained: direct computations lead to a system of
compatibilities that is very technical and impossible to apply in
practice.

We end the paper with an open question. A classical and very
difficult problem in the theory of groups is the following
(\cite[pg. 18]{AFG}): for two given abelian groups $A$ and $B$
describe and classify all groups $G$ which can be written as a
product $G = AB$. Having this question in mind as well as
\coref{itoass} we might ask:

\textbf{Question:} \emph{Let $V_0$ and $P_0$ be two abelian
algebras. Describe and classify all algebras $A$ containing $V_0$
and $P_0$ as subalgebras such that $A = V_0 + P_0$.}

A particular case of the question, corresponding to the additional
condition $P_0 \, \cap \, V_0 = \{0\}$, is the following:
\emph{for two given abelian algebras $V_0$ and $P_0$, describe and
classify all bicrossed products $P_0 \bowtie V_0$ of algebras} -
for details and the construction of the bicrossed product
associated to a matched pair of associative algebras we refer to
\cite{a-2013}.

\end{document}